\documentclass[12pt]{amsart}
\usepackage{amsfonts,amsmath,amscd,amssymb,paralist,amsthm}
\usepackage{mathrsfs}
\usepackage{amsxtra}
\usepackage[mathscr]{eucal}
\usepackage{graphics}
\usepackage{latexsym,graphicx,verbatim,nameref}
\usepackage[numbers,sort&compress]{natbib}
\usepackage[all,cmtip]{xy}
\usepackage{hyperref}
\usepackage{fancyhdr}
\usepackage{color}
\usepackage{tikz}
\usepackage{tikz-cd}
\usepackage{imakeidx}
\usetikzlibrary{arrows,scopes,matrix}
\numberwithin{equation}{section}

\theoremstyle{definition}

\makeatletter
\newtheorem*{rep@theorem}{\rep@title}
\newcommand{\newreptheorem}[2]{%
\newenvironment{rep#1}[1]{%
 \def\rep@title{#2 \ref{##1}}%
 \begin{rep@theorem}}%
 {\end{rep@theorem}}}
\makeatother

\newreptheorem{theorem}{Theorem}
\newreptheorem{corollary}{Corollary}
\newreptheorem{lemma}{Lemma}

\newtheorem{theorem}{Theorem}[section]
\newtheorem{corollary}[theorem]{Corollary}
\newtheorem{lemma}[theorem]{Lemma}

\newtheorem*{theorem*}{Theorem}
\newtheorem*{proposition*}{Proposition}

\newtheorem{definition}[theorem]{Definition}
\newtheorem{remark}[theorem]{Remark}

\newtheorem*{claim*}{Claim}
\newtheorem{conjecture}[theorem]{Conjecture}

\newtheorem*{conjecture*}{Conjecture}
\newtheorem*{observation*}{Observation}
\newtheorem*{question*}{Question}

\newcommand{\diag}{{\rm diag}}

\newcommand{\Cb}{{\mathbb C}}

\newcommand{\Fb}{{\mathbb F}}

\newcommand{\Nb}{{\mathbb N}}

\newcommand{\Qb}{{\mathbb Q}}
\newcommand{\Rb}{{\mathbb R}}

\newcommand{\Tb}{{\mathbb T}}

\newcommand{\Zb}{{\mathbb Z}}

\newcommand{\supp}{{\rm supp}}

\begin{document}

\title{Strongly Independent Matrices and Rigidity of $\times A$-Invariant Measures on $n$-Torus}

\date{\today}
\author{Huichi Huang}
\address{Huichi Huang, College of Mathematics and Statistics, Chongqing University, Chongqing, 401331, P.R. China}
\email{huanghuichi@cqu.edu.cn}
\thanks{Huichi Huang was partially supported by NSFC no. 11871119 and Chongqing Municipal Science and Technology Commission fund no. cstc2018jcyjAX0146.}
\author{Hanfeng Li}
\address{Hanfeng Li, Center of Mathematics, Chongqing University, Chongqing 401331, P.R. China
and
Department of Mathematics, SUNY at Buffalo, Buffalo, NY 14260-2900, USA}
\email{hfli@math.buffalo.edu}
\thanks{Hanfeng Li was partially supported by NSF grant DMS-1900746.}

\author{Enhui Shi}
\thanks{Enhui Shi was partially supported by NSFC no. 11771318 and no. 11790274.}
\address{Enhui Shi, School of mathematical science, Soochow University, Suzhou, 215006, P.R. China}
\email{ehshi@suda.edu.cn}

\author{Hui Xu}
\address{Hui Xu, School of mathematical science, Soochow University, Suzhou, 215006, P.R. China}
\email{20184007001@stu.suda.edu.cn}

\keywords{Strongly independent, ergodicity, mixing, Fourier coefficient, measure rigidity}
\subjclass[2020]{Primary 37A05, 37A25, 37A46, 43A05, 28C10, 12E05}
\begin{abstract}
We introduce the concept of strongly independent matrices over any field, and prove the existence of such matrices for certain fields and the non-existence for algebraically closed fields.  Then we apply strongly independent matrices over rational numbers to obtain a measure rigidity result for endomorphisms on $n$-torus.
\end{abstract}
\maketitle

\section{Introduction}

For an integer $m$, the $\times m$ map  $T_m$  on $\Tb=\{z\in\Cb:|z|=1\}$ is given by
$T_m(z)=z^m$ for all $z\in\Tb$.

H. Furstenberg proved that under the action of a non-lacunary multiplicative semigroup of positive integers on $\Tb$, a closed invariant subset of $\Tb$ containing a dense orbit  is either finite or the whole $\Tb$~\cite[Theorem IV.1]{F}. Here a multiplicative semigroup of positive integers is called {\bf non-lacunary} if it is not contained in any singly generated multiplicative semigroup. In other words a non-lacunary multiplicative semigroup of positive integers always contains two positive integers $p$ and $q$  with $\frac{\log p}{\log q}$ irrational~(we say that $p,q$ are non-lacunary).

Furthermore, Furstenberg  conjectured the following.

\begin{conjecture}[Furstenberg's Conjecture]\

An ergodic invariant Borel probability  measure on $\Tb$  under the action of a non-lacunary multiplicative semigroup of positive integers is either finitely supported or the Lebesgue measure.
\end{conjecture}

The first breakthrough of Furstenberg's conjecture was achieved by R. Lyons.

\begin{theorem}\cite[Theorem 1]{L}\

Suppose $p,q\geq 2$ are two relatively prime integers. If  a non-atomic $\times p,\times q$-invariant  Borel probability measure $\mu$ on $\Tb$ is $T_p$-exact, then it is the Lebesgue measure. Here $\mu$ is $T_p$-exact means that $(\Tb, \mathscr{B},\mu, T_p)$ has no nontrivial zero entropy factor.
\end{theorem}

This result was improved by D. J. Rudolph under the assumption that $p$ and $q$ are coprime and an extra positive entropy condition~\cite[Theorem 4.9]{R} and later by A. S. A. Johnson ~\cite[Theorem A]{J} under the assumption that $p,q$ are non-lacunary and the positive entropy condition.

\begin{theorem}[Rudolph-Johnson's Theorem]\

Suppose $p$ and $q$ are non-lacunary positive integers greater than 1. If $\mu$ is an ergodic $\times p,\times q$-invariant  Borel probability measure on $\Tb$ such that $T_p$ or $T_q$ has positive measure entropy with respect to $\mu$, then $\mu$ is the Lebesgue measure.
\end{theorem}

One may consult~\cite{KK, KS1, KS2} for the extensions of above results to automorphisms on $n$-torus with $n\geq 2$.

Recently, the first named author obtained the following rigidity theorem.

\begin{theorem}\cite[Theorem 1.5]{H}\

Let $p$ be a nonzero integer. The Lebesgue measure is the unique non-atomic $\times p$-invariant Borel probability measure on $\Tb$ satisfying one of the following:
\begin{enumerate}
  \item it is ergodic and there exist a nonzero integer $l$ and a F{\o}lner sequence $\Sigma=\{F_m\}_{m=1}^\infty$ in $\mathbb{N}$
  such that $\mu$  is $\times(p^j+l)$-invariant for all $j$ in some $E\subseteq \mathbb{N}$ with upper density $\overline{D}_{\Sigma}(E)$~(see Definition~\ref{density}) equal to $1$;
  \item  it is weakly mixing and there exist a nonzero integer $l$ and a F{\o}lner sequence $\Sigma=\{F_m\}_{m=1}^\infty$ in $\mathbb{N}$
  such that $\mu$  is $\times(p^j+l)$-invariant for all $j$ in some $E\subseteq \mathbb{N}$ with $\overline{D}_{\Sigma}(E)>0$;
  \item  it is strongly mixing and there exist a nonzero integer $l$ and an infinite set $E\subseteq \mathbb{N}$
  such that $\mu$  is $\times(p^j+l)$-invariant for all $j$ in  $E$.
\end{enumerate}
Moreover, a $\times p$-invariant Borel probability measure satisfying (2) or (3) is either a Dirac measure or the Lebesgue measure.
\end{theorem}

In this paper, we introduce so-called strongly independent matrices over a field $\Fb$, and use strongly independent matrices over the rational field $\Qb$ to extend the above measure rigidity results to endomorphisms on  $\mathbb T^n=\{(z_1,\cdots,z_n)\in\Cb^n: |z_1|=\cdots=|z_n|=1\}$.

We say that an $n$-tuple  $(B_1, B_2, \cdots, B_n)$ of $n\times n$ matrices over a field $\Fb$ is  \textbf{strongly independent over $\Fb$} if for any nonzero column vector $v$ in $\Fb^{n\times 1}$, the vectors $B_1v, B_2v, \cdots, B_nv$ are linearly independent over $\Fb$. A nonzero matrix $B$ in $M_n(\Fb)$ is called strongly independent over $\Fb$ if the $n$-tuple $(I_n, B,  \cdots, B^{n-1})$ is strongly independent over $\Fb$.

The next main theorem illustrates the existence of an abundance of strongly independent matrices.

\begin{theorem}\label{strongly independent matrix}
A nonzero matrix $B$ in $M_n(\Fb)$ is strongly independent over $\Fb$ iff the characteristic polynomial of $B$ is irreducible in $\Fb[t]$.
\end{theorem}

The above shows existence of strongly independent matrices over certain fields, say, the field of rational numbers $\Qb$.  However over some fields, there are no strongly independent matrices.

\begin{theorem}
\label{non-existence}
If $\Fb$ is an algebraically closed field, then there are no strongly independent $n$-tuples in $M_n(\Fb)$ for $n\geq 2$.
\end{theorem}

We shall identify $\Rb^n/\Zb^n$ with the n-torus $\Tb^n$ naturally via
$$\Rb^n/\Zb^n \ni(x_1, x_2, \cdots, x_n)+\Zb^n\mapsto (e^{2\pi i x_1}, e^{2\pi i x_2}, \cdots, e^{2\pi ix_n})\in \Tb^n$$
for $(x_1, x_2, \cdots, x_n)\in \Rb^n$. Let $A$ be a matrix in  $M_n(\Zb)$. The $\times A$ map on $\Tb^n$ is defined by
$T_{A}: \Rb^n/\Zb^n\to \Rb^n/\Zb^n$
\begin{displaymath}
T_{A}((x_1, x_2,\cdots,x_n)+\Zb^n)=(x_1, x_2,\cdots,x_n)A+\Zb^n,
\end{displaymath}
for $(x_1, x_2,\cdots,x_n)$ in $\Rb^n$.

\begin{theorem}\label{rigidity on torus}
Let $A$ be in $M_n(\Zb)$. Suppose that $\mu$ is a $\times A$-invariant Borel probability measure on $\Tb^n$ satisfying one of the following:
\begin{enumerate}
  \item  it is ergodic and there exists an $n$-tuple  $(B_1, B_2, \cdots, B_n)$ of matrices in $M_n(\Zb)$  strongly independent over $\Qb$ and a F{\o}lner sequence
   $\Sigma=\{F_m\}_{m=1}^\infty$ in $\mathbb{N}$ such that $\mu$  is $\times(A^j+B_i)$-invariant for all $j$ in some $E\subseteq \mathbb{N}$ with upper density $\overline{D}_{\Sigma}(E)=1$ and all $i=1, 2, \cdots, n$;
  \item    it is weakly mixing and there exists an $n$-tuple  $(B_1, B_2, \cdots, B_n)$ of matrices in $M_n(\Zb)$  strongly independent over $\Qb$ and a F{\o}lner sequence $\Sigma=\{F_m\}_{m=1}^\infty$ in $\mathbb{N}$ such that $\mu$  is $\times(A^j+B_i)$-invariant for all $j$ in some $E\subseteq \mathbb{N}$ with $\overline{D}_{\Sigma}(E)>0$ and all $i=1, 2, \cdots, n$;
  \item it is strongly mixing and there exists an  $n$-tuple  $(B_1, B_2, \cdots, B_n)$ of matrices in $M_n(\Zb)$  strongly independent over $\Qb$ and an infinite set $E\subseteq \mathbb{N}$
  such that $\mu$  is $\times(A^j+B_i)$-invariant for all $j$ in  $E$ and all $i=1, 2, \cdots, n$.
\end{enumerate}
Then $\mu$ is either finitely supported or the Lebesgue measure.

Moreover, a $\times A$-invariant Borel probability measure satisfying (2) or (3) is either a Dirac measure or the Lebesgue measure.
\end{theorem}

Consequently  there exist ``very small'' semigroups acting on $\Tb^n$ such that the Lebesgue measure is the unique non-atomic invariant measure.

\begin{corollary} \label{C-rigidity}
There exists an abelian multiplicative semigroup $S\subseteq M_n(\Zb)$ acting on $\Tb^n$  such that  the Lebesgue measure is the unique non-atomic Borel probability measure on $\Tb^n$ which is both invariant under $\times A$ for all $A$ in $S$ and ergodic under $\times B$ for some $B$ in $S$.
\end{corollary}

The paper is organized as follows. We lay down some definitions and notations in Section 2. Theorem~\ref{strongly independent matrix} and Theorem~\ref{non-existence} are proved in Section 3. In Section 4, we characterize mixing properties of Borel probability measures on $\Tb^n$ in terms of their Fourier coefficients. Finally we establish Theorem~\ref{rigidity on torus} in Section 5.

\section{Preliminaries}

Denote the set of nonnegative integers by $\mathbb{N}$,  and the cardinality of a set $E$ by $|E|$.

For a ring $R$, denote by $M_n(R)$ the ring of $n\times n$ square matrices with entries in $R$.  Denote by $GL_n(R)$  the group  of invertible elements in $M_n(R)$. For a field $\Fb$, denote by $\overline{\Fb}$ its algebraic closure. For any $A\in M_n(\Fb)$, denote by $P_A(t)$ the characteristic polynomial $\det(t I_n-A)$ of $A$ in $\Fb[t]$.

For a nonempty set $Z$, denote by $Z^n$ the set of row vectors of length $n$ with coordinates in $Z$, and by $Z^{n\times 1}$ the set of column vectors of length $n$ with coordinates in $Z$.

Within this paper, a measure on a compact metrizable space $X$ always means a Borel probability measure. Denote by $C(X)$ the space of complex-valued continuous functions on $X$.

\begin{definition}
A \textbf{F{\o}lner sequence} in $\mathbb{N}$ is a sequence $\Sigma=\{F_m\}_{m=1}^\infty$ of nonempty finite subsets of $\mathbb{N}$ satisfying
\begin{displaymath}
\lim_{m\rightarrow \infty}\frac{|(F_m+m')\Delta F_m|}{|F_m|}=0
\end{displaymath}
for every $m'$ in $\mathbb{N}$. Here $\Delta$ stands for the symmetric difference.
\end{definition}

\begin{definition}
\label{density}
Let $\Sigma=\{F_m\}_{m=1}^\infty$ be a sequence of nonempty finite subsets of $\mathbb{N}$. For a subset $E$ of $\mathbb{N}$,
the \textbf{upper density} $ \overline{D}_{\Sigma}(E)$
is given by
$$\overline{D}_{\Sigma}(E):=\limsup_{m\rightarrow \infty}\frac{|E\cap F_m|}{|F_m|}.$$
\end{definition}

\begin{definition}
For $k=\left[\begin{matrix} k_1 \\ k_2 \\ \vdots \\ k_n  \end{matrix}\right]\in\Zb^{n\times 1}$ and $z=(z_1,z_2,\cdots,z_n)\in \Tb^n$,
use $z^k$ to denote $z_1^{k_1}z_2^{k_2}\cdots z_n^{k_n}$, and the \textbf{Fourier coefficient} $\hat{\mu}(k)$ of a measure $\mu$ on $\Tb^n$ is defined by
$$\hat{\mu}(k)=\int_{\Tb^n} z^k\,d\mu(z).$$ 
\end{definition}

For a measure $\mu$ on a compact metrizable space $X$,  if $\mu(\{x\})>0$ for some $x$ in $X$, then $x$ is called an \textbf{atom} for $\mu$. A measure with no atoms is called \textbf{non-atomic}.

For a continuous map $T: X\rightarrow X$, a measure $\mu$ on $X$ is called $T$-\textbf{invariant}  if $\mu(E)=\mu(T^{-1}E)$ for every Borel subset $E$ of $X$. For $A$ in $M_n(\Zb)$, we call a measure $\mu$ on $\Tb^n$ $\times A$-invariant if $\mu$ is $T_A$-invariant.

A $T$-invariant measure $\mu$ is called \textbf{ergodic} if every Borel subset $E$ with $T^{-1}E=E$ satisfies $\mu(E)=0$ or $1$.  A measure $\mu$ is called \textbf{weakly mixing} if $\mu\times \mu$ is an ergodic $T\times T$-invariant measure on $X\times X$, and it is called \textbf{strongly mixing} if $\lim_{j\rightarrow \infty}\mu(T^{-j}E\cap F)=\mu(E)\mu(F)$ for all Borel subsets $E,F$ of $X$.

\section{Existence and non-existence of Strongly Independent Matrices over certain fields}

In this section, we prove Theorems~\ref{strongly independent matrix} and \ref{non-existence}, which illustrate  that the existence of strongly independent matrices over a field $\Fb$ depends on algebraic properties of $\Fb$.

\begin{definition}\label{independent}
For a field $\Fb$, we call an $n$-tuple $(B_1, B_2, \cdots, B_n)$ of matrices in $M_n(\Fb)$ \textbf{strongly independent over $\Fb$} if for any nonzero $v$ in $\Fb^{n\times 1}$, the vectors $B_1v, B_2v,\cdots, B_nv$ are linearly independent over $\Fb$. We call a nonzero matrix $B$ in $M_n(\Fb)$ \textbf{strongly independent over $\Fb$} if the $n$-tuple $(I_n, B, \cdots, B^{n-1})$ is strongly independent over $\Fb$.
\end{definition}

\begin{lemma} \label{L-strong vs linear}
Let $B_1, \cdots, B_n\in M_n(\Fb)$. The tuple $(B_1, \cdots, B_n)$ is strongly independent over $\Fb$ iff for any nonzero $(u_1, \cdots, u_n)\in \Fb^n$ the matrix $\sum_{j=1}^n u_jB_j$ is invertible.
\end{lemma}
\begin{proof} The tuple $(B_1, \cdots, B_n)$ is strongly independent over $\Fb$ iff for any nonzero $v\in \Fb^{n\times 1}$ the vectors $B_1v, \cdots, B_nv$ are linearly independent, iff for any nonzero $v\in \Fb^{n\times 1}$ and any nonzero $(u_1, \cdots, u_n)\in \Fb^n$ the vector $\sum_{j=1}^n u_jB_jv$ is nonzero, iff for any nonzero $(u_1, \cdots, u_n)\in \Fb^n$ the matrix $\sum_{j=1}^n u_jB_j$ is invertible.
\end{proof}

\begin{proof}[Proof of Theorem~\ref{strongly independent matrix}]\

Suppose that $P_B(t)$ is not irreducible in $\Fb[t]$. We have $P_B(t)=f(t)g(t)$ for some $f, g\in \Fb[t]$ with $1\le \deg(f), \deg(g)\le n-1$. Then $0=P_B(B)=f(B)g(B)$ by Hamilton-Cayley Theorem \cite[Theorem XIV.3.1]{Lang}, whence at least one of $f(B)$ and $g(B)$ is not invertible. By Lemma~\ref{L-strong vs linear} we conclude that $B$ is not strongly independent over $\Fb$.

Now assume that $P_B(t)$ is irreducible in $\Fb[t]$. Denote by $D$ the Jordan canonical form of $B$. That is, $D\in M_n(\overline{\Fb})$ and there is some invertible $W\in M_n(\overline{\Fb})$ satisfying $WB=DW$ and
$$ D=\left[\begin{matrix} D_1 & & \\ & \ddots & \\ & & D_k \end{matrix}\right]$$
for some positive integer $k$ such that each $D_i$ is in $M_{m_i}(\overline{\Fb})$ of the form
$$\left[\begin{matrix} \lambda_i &  & \\  1 & \lambda_i &  & \\ &  \ddots & \ddots & & \\ & & 1 & \lambda_i &  \\ & & & 1 & \lambda_i \end{matrix}\right]$$
for some $\lambda_i\in \overline{\Fb}$ and positive integer $m_i$. Then $P_B(t)=\prod_{i=1}^k(t-\lambda_i)^{m_i}$, whence $P_B(\lambda_i)=0$ for every $1\le i\le k$. Since $P_B(t)$ is irreducible in $\Fb[t]$, it follows that for any nonzero $f(t)\in \Fb[t]$ of degree at most $n-1$, one has $f(\lambda_i)\neq 0$ for every $1\le i\le k$.

Let $(u_1, \cdots, u_n)$ be a nonzero vector in $\Fb^n$. Then $f(t)=\sum_{j=1}^nu_jt^{j-1}\in \Fb[t]$ is  nonzero and has degree at most $n-1$. Thus $f(\lambda_i)\neq 0$ for every $1\le i\le k$. It follows that $f(D)$ is invertible, whence $u_1I_n+u_2B+\cdots+u_nB^{n-1}=f(B)=W^{-1}f(D)W$ is invertible.
By Lemma~\ref{L-strong vs linear} we conclude that $B$ is  strongly independent over $\Fb$.
\end{proof}

\begin{remark}\label{remark: existence}
Suppose $f(t)=t^n+a_1t^{n-1}+\cdots+a_{n-1}t+a_n$ is an irreducible polynomial in $\Fb[t]$. Define $B\in M_n(\Fb)$ as
  \begin{displaymath}
  \begin{bmatrix}
0&0&0&\cdots&0&-a_n\\
1&0&0&\cdots&0&-a_{n-1}\\
0&1&0&\cdots&0&-a_{n-2}\\
\vdots&\vdots&\vdots&\cdots&\vdots&\vdots\\
0&0&0&\cdots&0&-a_2\\
0&0&0&\cdots&1&-a_1
\end{bmatrix}.
\end{displaymath}
Then  $P_B(t)=f(t)$~\cite[Definition on page 173 and Lemma 7.17]{Rotman}. By Theorem~\ref{strongly independent matrix}, the matrix $B$ is strongly independent over $\Fb$.

For any $n\geq1$, by Eisenstein's criterion~\cite[Theorem IV.3.1]{Lang}, there exist infinitely many monic  polynomials of degree n in $\Zb[t]$, which are irreducible in $\Qb[t]$ (for example $t^n+p$ for any prime number $p$ in $\Zb$). Theorem~\ref{strongly independent matrix} illustrates that for $n\ge 2$ there are infinitely many $n$-tuples of the form $(I_n, B, \cdots, B^{n-1})$ in $M_n(\Zb)$ strongly independent over $\Qb$.
\end{remark}

Next we prove Theorem~\ref{non-existence} which gives the non-existence of strongly independent matrices over algebraically closed fields.

\begin{proof}[Proof of Theorem~\ref{non-existence}]\

For any matrices $B_1, B_2,\cdots, B_n$ in $M_n(\Fb)$, taking $z=\begin{bmatrix}
                                                          z_1 \\
                                                          z_2 \\
                                                          \vdots \\
                                                          z_n
                                                        \end{bmatrix}$, the polynomial $f(z_1, z_2, \cdots, z_n)=\det\begin{bmatrix} B_1z & B_2z & \cdots & B_nz \end{bmatrix}$ is in $\Fb[z_1, z_2, \cdots, z_n]$. Now $f(z_1, z_2, \cdots, z_n)=0$
always has a nonzero solution $\tilde{z}$ in $\Fb^{n\times 1}$ since $\Fb$ is algebraically closed and $n\ge 2$.
\end{proof}

\section{Fourier Coefficients of Ergodic, Weakly Mixing and Strongly Mixing measures on $\Tb^n$ }

In this section we prove Theorem~\ref{fourier coefficient}, characterizing the mixing properties of measures on $\Tb^n$ under $\times A$ map via their Fourier coefficients.

\begin{theorem}
\label{fourier coefficient}
Let $A\in M_n(\Zb)$ and let $\Sigma=\{F_m\}_{m=1}^\infty$ be  a  F{\o}lner sequence in $\mathbb{N}$. The following are true.
\begin{enumerate}
\item A measure $\mu$ on $\Tb^n$ is an ergodic $\times A$-invariant measure iff
\begin{align} \label{eq 4.3}
\lim_{m\rightarrow \infty}\frac{1}{|F_m|}\sum_{j\in F_m} \hat{\mu}(A^jk+l)=
\hat{\mu}(k)\hat{\mu}(l)
\end{align}
for
all $k,l$ in $\Zb^{n\times 1}$.
\item    A measure $\mu$ on $\Tb^n$ is a weakly mixing $\times A$-invariant measure iff
\begin{align} \label{eq 5.4}
\lim_{m\rightarrow \infty}\frac{1}{|F_m|}\sum_{j\in F_m} |\hat{\mu}(A^jk+l)-\hat{\mu}(k)\hat{\mu}(l)|^2=0
\end{align}
for
all $k,l$ in $\Zb^{n\times 1}$.
\item   A measure $\mu$ on $\Tb^n$ is a strongly mixing $\times A$-invariant measure iff
\begin{align} \label{eq 5.5}
\lim_{j\rightarrow \infty}\hat{\mu}(A^jk+l)=\hat{\mu}(k)\hat{\mu}(l)
\end{align}
for all $k,l$ in $\Zb^{n\times 1}$.
\end{enumerate}
\end{theorem}

To prove Theorem~\ref{fourier coefficient} we need to make some preparations.

\begin{lemma}
\label{lemma:inv}
Let $A\in M_n(\Zb)$. A measure $\mu$ on $\Tb^n$ is $\times A$-invariant iff $\hat{\mu}(k)=\hat{\mu}(Ak)$ for all $k$ in $\Zb^{n\times 1}$.
\end{lemma}
\begin{proof}
A measure $\mu$ on $\Tb^n$ is $\times A$-invariant iff $\int_{\Tb^n} f(T_A z)\, d\mu(z)=\int_{\Tb^n} f(z)\,d\mu(z)$ for all $f$ in $C(\Tb^n)$ \cite[Theorem 6.8]{Walters} iff
$\int_{\Tb^n} f(T_A z)\, d\mu(z)=\int_{\Tb^n} f(z)\,d\mu(z)$ for all $f$ in a dense subset of $C(\Tb^n)$ iff $\int_{\Tb^n} f(T_A z)\, d\mu(z)=\int_{\Tb^n} f(z)\,d\mu(z)$ for $f(z)=z^k$ for all $k$ in $\Zb^{n\times 1}$ since the linear span of $z^k$'s is dense in $C(\Tb^n)$. Note that $\int_{\Tb^n} (T_A z)^k\, d\mu(z)=\hat{\mu}(Ak)$ for  all $k$ in $\Zb^{n\times 1}$.
\end{proof}

\begin{lemma}\label{supp}
Let $\mu$ be a  measure on $\Tb^n$. For any $k$ in $\Zb^{n\times 1}$, if $\hat{\mu}(k)=1$ then the support of $\mu$
\begin{displaymath}
{\rm supp}(\mu)\subseteq \{z\in\Tb^n:~z^k=1\}.
\end{displaymath}
\end{lemma}

\begin{proof}
Since  $\hat{\mu}(k)=1$, by the definition of $\hat{\mu}(k)$, we have
\begin{displaymath}
\int_{\Tb^n} z^k d\mu(z)=1.
\end{displaymath}
Thus $\int_{\Tb^n} {\rm Re}(z^k)d\mu(z)=1$. Therefore,
\begin{displaymath}
\int_{\Tb^n}| z^k-1|^2 d\mu(z)=\int_{\Tb^n}(2-2{\rm Re}(z^k) ) d\mu(z)=0.
\end{displaymath}
Hence, ${\rm supp}(\mu)\subseteq \{z\in\Tb^n:~z^k=1\}$.
\end{proof}

\begin{lemma}\label{atomic}
Let $\mu$ be a measure on $\Tb^n$. Let an $n$-tuple $(B_1, B_2, \cdots, B_n)$ of matrices in $M_n(\Zb)$ be strongly independent over $\Qb$.  If there is  some nonzero $k$ in $\Zb^{n\times 1}$ such that $\hat{\mu}(B_ik)=1$ for every $1\leq i\leq n$, then $\mu$ is finitely supported.
\end{lemma}

\begin{proof}
 Let $L=\begin{bmatrix}
 B_1k &
 \cdots &
 B_nk
 \end{bmatrix}\in M_n(\Zb)$.
 Since $B_1k, B_2k, \cdots, B_nk$ are linearly independent over $\Qb$, the matrix $L$ is in $GL_n(\Qb)$. Write $L$ as $(L_{i, j})_{1\le i, j\le n}$ and put $M=\sum_{1\le i, j\le n}|L_{i, j}|$.
 By Lemma \ref{supp}, the support of $\mu$, ${\rm supp}(\mu)$, is a subset of $\bigcap_{i=1}^n\{z\in\Tb^n:~z^{B_ik}=1\}$. That is,
\begin{align*}
{\rm supp}(\mu)&\subseteq \bigcap_{i=1}^n\{x+\Zb^n:~x\in [0, 1)^n, xB_ik\in \Zb\}\\
&=\{x+\Zb^n:~x\in [0, 1)^n, xL\in \Zb^n\}\\
&\subseteq \{wL^{-1}+\Zb^n:~w\in [-M,M]^n\cap\Zb^n\}.
\end{align*}
Note that $\{wL^{-1}:~w\in [-M,M]^n\cap\Zb^n\}$ is finite, so is ${\rm supp}(\mu)$.
\end{proof}

We need the following Lemma~\cite[Lemma 4.2]{H} which is a special case of the mean ergodic theorem for amenable semigroups~\cite[Theorem 1]{B}.

\begin{lemma}\label{ergodicmean}
For a compact metrizable space $X$ and a continuous map $T:X\to X$, if $\nu$ is an ergodic $T$-invariant measure on $X$, then for every F{\o}lner sequence $\{F_m\}_{m=1}^\infty$ in $\mathbb{N}$, one has
\begin{displaymath}
\lim_{m\rightarrow \infty} \frac{1}{|F_m|}\sum_{j\in F_m}f\circ T^j=\int_X f d\nu
\end{displaymath}
for every $f\in L^2(X,\nu)$ (note that the identity holds with respect to $L^2$-norm). Consequently
\begin{equation}
\label{eq 4.1}
\lim_{m\rightarrow \infty} \frac{1}{|F_m|}\sum_{j\in F_m}\int_X f(T^j x)g(x) d\nu(x)=\int_X f d\nu \int_X g d \nu
\end{equation}
for all $f,g$ in $L^2(X,\nu)$.
\end{lemma}

\begin{proof}[Proof of Theorem~\ref{fourier coefficient}]\

For any Borel subset $E$ of $\Tb^n$, write $1_E$ for the characteristic function of $E$.

(1) Suppose $\mu$ is an ergodic $\times A$-invariant measure on $\Tb^n$.
  Applying Lemma~\ref{ergodicmean} for $X=\Tb^n, T=T_A$ and $\nu=\mu$, we have
\begin{eqnarray}
\label{eq 4.2}
\lim_{m\rightarrow \infty} \frac{1}{|F_m|}\sum_{j\in F_m}\int_{\Tb^n} f(T_A^j z)g(z) d\mu(z)=\int_{\Tb^n} f d\mu \int_{\Tb^n} g d \mu
\end{eqnarray}
for all continuous functions $f,g$ on ${\Tb^n}$. Letting $f(z)=z^k$ and $g(z)=z^l$ for $z$ in $\Tb^n$ and $k,l$ in $\Zb^{n\times 1}$, we obtain \eqref{eq 4.3},
which is the necessity.

Now assume that \eqref{eq 4.3} holds for
all $k,l$ in $\Zb^{n\times 1}$.

Let $k\in \Zb^{n\times 1}$.
Letting $l=0$ in \eqref{eq 4.3}, we get $\lim_{m\rightarrow \infty}\frac{1}{|F_m|}\sum_{j\in F_m} \hat{\mu}(A^jk)=\hat{\mu}(k)$.
Replacing $k$ by $Ak$, we also have
\begin{eqnarray*}
\hat{\mu}(Ak)=\lim_{m\rightarrow \infty}\frac{1}{|F_m|}\sum_{j\in F_m} \hat{\mu}(A^{j+1}k)
=\lim_{m\rightarrow \infty}\frac{1}{|F_m|}\sum_{j\in F_m+1} \hat{\mu}(A^jk).
\end{eqnarray*}
Then
\begin{align*}
|\hat{\mu}(Ak)-\hat{\mu}(k)|&=\lim_{m\rightarrow \infty}\frac{1}{|F_m|}\bigg|\sum_{j\in F_m+1}\hat{\mu}(A^jk)-\sum_{j\in F_m}\hat{\mu}(A^jk)\bigg|\\
&\le \lim_{m\rightarrow \infty}\frac{|(F_m+1)\Delta F_m|}{|F_m|}=0,
\end{align*}
whence $\hat{\mu}(Ak)=\hat{\mu}(k)$.
By Lemma~\ref{lemma:inv}, we get that $\mu$ is $\times A$-invariant.

From \eqref{eq 4.3} we see that \eqref{eq 4.2} is true for all  $f(z)=z^{k}$ and $g(z)=z^{l}$ with $k,l$ in $\Zb^{n\times 1}$. By linearity, \eqref{eq 4.2} is also true for all $f, g$ in the linear span $V$ of $z^k$ for all $k\in \Zb^{n\times 1}$. Since $V$ is dense in $L^2(\Tb^n,\mu)$, \eqref{eq 4.2} is true for all $f,g\in L^2(\Tb^n,\mu)$. For any Borel subset $E$ of $\Tb^n$ satisfying $T^{-1}_A E=E$, taking $f=g=1_E$ in \eqref{eq 4.2}, we get $\mu(E)=\mu(E)^2$. Hence $\mu$ is ergodic.

(2) Suppose $\mu$ is a weakly mixing $\times A$-invariant measure on $\Tb^n$, which means $\mu\times\mu$ is an ergodic $T_A\times T_A$-invariant measure on $\Tb^{2n}$. Let $k, l\in \Zb^{n\times 1}$. Taking $f(z', z'')=(z')^{k}(z'')^{-k}$ and $g(z', z'')=(z')^{l}(z'')^{-l}$ in \eqref{eq 4.1} of Lemma~\ref{ergodicmean} with $X=\Tb^n\times \Tb^n, T=T_A\times T_A$ and $\nu=\mu\times\mu$, we get
\begin{eqnarray} \label{eq 5.3}
\lim_{m\rightarrow\infty}\frac{1}{|F_m|}\sum_{j\in F_m}|\hat{\mu}(A^jk+l)|^2=|\hat{\mu}(k)|^2|\hat{\mu}(l)|^2.
\end{eqnarray}
Taking $f(z', z'')=(z')^{k}$ and $g(z', z'')=(z')^{l}$ in \eqref{eq 4.1} of Lemma~\ref{ergodicmean} with $X=\Tb^n\times \Tb^n, T=T_A\times T_A$ and $\nu=\mu\times\mu$, we also get
\begin{eqnarray} \label{eq 5.20}
\lim_{m\rightarrow\infty}\frac{1}{|F_m|}\sum_{j\in F_m}\hat{\mu}(A^jk+l)=\hat{\mu}(k)\hat{\mu}(l).
\end{eqnarray}
Since
\begin{eqnarray*}
&&|\hat{\mu}(A^jk+l)-\hat{\mu}(k)\hat{\mu}(l)|^2\\
&=& |\hat{\mu}(A^jk+l)|^2+|\hat{\mu}(k)|^2|\hat{\mu}(l)|^2-\hat{\mu}
      (A^jk+l)\overline{\hat{\mu}(k)}\overline{\hat{\mu}(l)}- \overline{\hat{\mu}(A^jk+l)}\hat{\mu}(k)\hat{\mu}(l),
\end{eqnarray*}
we have
\begin{align*}
&\lim_{m\rightarrow\infty}\frac{1}{|F_m|}\sum_{j\in F_m}|\hat{\mu}(A^jk+l)-\hat{\mu}(k)\hat{\mu}(l)|^2 \\
&=\lim_{m\rightarrow\infty}\frac{1}{|F_m|}\sum_{j\in F_m}[|\hat{\mu}(A^jk+l)|^2+|\hat{\mu}(k)|^2|\hat{\mu}(l)|^2  \\
& \quad \quad \quad \quad \quad \quad -\hat{\mu}(A^jk+l)\overline{\hat{\mu}(k)}\overline{\hat{\mu}(l)}-\overline{\hat{\mu}(A^jk+l)}\hat{\mu}(k)\hat{\mu}(l)]\\
&\overset{\eqref{eq 5.3}, \eqref{eq 5.20}}=|\hat{\mu}(k)|^2|\hat{\mu}(l)|^2+|\hat{\mu}(k)|^2|\hat{\mu}(l)|^2-|\hat{\mu}(k)|^2|\hat{\mu}(l)|^2-
|\hat{\mu}(k)|^2|\hat{\mu}(l)|^2=0.
\end{align*}
This proves the necessity.

Conversely, suppose that \eqref{eq 5.4} holds
for all $k,l\in\Zb^{n\times 1}$.

Note that $T_A\times T_A=T_{\diag(A, A)}$ on $\Tb^n\times \Tb^n=\Tb^{2n}$.
In order to prove that $\mu\times \mu$ is an ergodic $T_A\times T_A$-invariant measure on $\Tb^n\times \Tb^n$, by part (1) it suffices to show that
\begin{align*}
\lim_{m\rightarrow \infty}\frac{1}{|F_m|}\sum_{j\in F_m}\widehat{\mu\times \mu}(\begin{bmatrix}
                                                          A & \\
                                                          & A
                                                        \end{bmatrix}^j\begin{bmatrix}
                                                          k' \\
                                                          k''
                                                        \end{bmatrix}+\begin{bmatrix}
                                                          l' \\
                                                          l''
                                                        \end{bmatrix})=\widehat{\mu\times \mu}(\begin{bmatrix}
                                                          k' \\
                                                          k''
                                                        \end{bmatrix})\widehat{\mu\times \mu}(\begin{bmatrix}
                                                          l' \\
                                                          l''
                                                        \end{bmatrix})
\end{align*}
for all $k', k'', l', l''\in \Zb^{n\times 1}$.
Note that
$$\widehat{\mu\times \mu}(\begin{bmatrix}
                                                          u \\
                                                          v
                                                        \end{bmatrix})=\hat{\mu}(u)\hat{\mu}(v)
                                                        $$
for all $u, v\in \Zb^{n\times 1}$. Thus it suffices to show
\begin{eqnarray*}
\lim_{m\rightarrow \infty}\frac{1}{|F_m|}\sum_{j\in F_m}\hat{\mu}(A^jk'+l')\hat{\mu}(A^jk''+l'')=
\hat{\mu}(k')\hat{\mu}(k'')\hat{\mu}(l')\hat{\mu}(l'')
\end{eqnarray*}
for all $k',k'',l',l''\in \Zb^{n\times 1}$.

Note that
\begin{eqnarray*}
&&|\hat{\mu}(A^jk'+l')\hat{\mu}(A^jk''+l'')-
\hat{\mu}(k')\hat{\mu}(k'')\hat{\mu}(l')\hat{\mu}(l'')|\\
&\leq&|\hat{\mu}(A^jk'+l')[\hat{\mu}(A^jk''+l'')-
\hat{\mu}(k'')\hat{\mu}(l'')]|
+|[\hat{\mu}(A^jk'+l')-\hat{\mu}(k')\hat{\mu}(l')]
 \hat{\mu}(k'')\hat{\mu}(l'')|\\
&\leq&|\hat{\mu}(A^jk''+l'')-\hat{\mu}(k'')(l'')|
 +|\hat{\mu}(A^jk'+l')-\hat{\mu}(k')\hat{\mu}(l')|
\end{eqnarray*}
for all $k',k'',l',l''\in \Zb^{n\times 1}$,
whence
\begin{eqnarray*}
&&\lim_{m\rightarrow \infty}\frac{1}{|F_m|}\sum_{j\in F_m} |\hat{\mu}(A^jk'+l')\hat{\mu}(A^jk''+l'')-
\hat{\mu}(k')\hat{\mu}(k'')\hat{\mu}(l')\hat{\mu}(l'')|^2\\
&\leq& \lim_{m\rightarrow \infty}\frac{1}{|F_m|}\sum_{j\in F_m}[|\hat{\mu}(A^jk''+l'')-\hat{\mu}(k'')\hat{\mu}(l'')|
+|\hat{\mu}(A^jk'+l')-\hat{\mu}(k')\hat{\mu}(l')|]^2\\
&\leq& 2\lim_{m\rightarrow \infty}\frac{1}{|F_m|}\sum_{j\in F_m}[|\hat{\mu}(A^jk''+l'')-\hat{\mu}(k'')\hat{\mu}(l'')|^2
+|\hat{\mu}(A^jk'+l')-\hat{\mu}(k')\hat{\mu}(l')|^2]\\
&\overset{\eqref{eq 5.4}}=&0,
\end{eqnarray*}
where in the second inequality we use  $(a+b)^2\leq 2(a^2+b^2)$ for all real numbers $a, b$.
This proves the sufficiency.

(3) Suppose $\mu$ is strongly mixing, which means that $\lim_{j\rightarrow\infty}\mu(T_A^{-j}E\cap F)=\mu(E)\mu(F)$ for all Borel subsets $E,F$  of $\Tb^n$.
Then
\begin{eqnarray*}
\lim_{j\rightarrow \infty} \int_{\Tb^n} 1_E(T_A^jz)1_F(z)d\mu(z)=\int_{\Tb^n}1_Ed\mu \int_{\Tb^n}1_Fd\mu
\end{eqnarray*}
for all Borel subsets $E,F$ of $\Tb^n$.
Since the linear combinations of characteristic functions are dense in $L^2(\Tb^n,\mu)$, we have
\begin{eqnarray*}
\lim_{j\rightarrow \infty} \int_{\Tb^n} f(T_A^jz)g(z)d\mu(z)=\int_{\Tb^n}fd\mu\int_{\Tb^n}g\,d\mu
\end{eqnarray*}
for all $f,g\in C(\Tb^n)$. In particular, taking $f(z)=z^k$ and $g(z)=z^l$, we obtain \eqref{eq 5.5}
for all $k,l$ in $\Zb^{n\times 1}$. This proves the necessity.

On the other hand, suppose a measure $\mu$ on $\Tb^n$ satisfies \eqref{eq 5.5} for all $k, l\in \Zb^{n\times 1}$. Let $l=0$ and replace $k$ by $Ak$. Then
\begin{displaymath}
\hat{\mu}(Ak)=\lim_{j\rightarrow \infty}\hat{\mu}(A^{j+1}k)=\lim_{j\rightarrow \infty}\hat{\mu}(A^jk)=\hat{\mu}(k)
\end{displaymath}
for all $k\in\Zb^{n\times 1}$. Hence $\mu$ is $\times A$-invariant in view of Lemma~\ref{lemma:inv}. From \eqref{eq 5.5} we have
\begin{eqnarray*}
\lim_{j\rightarrow \infty} \int_{\Tb^n} f(T_A^jz)g(z)d\mu(z)=\int_{\Tb^n}fd\mu \int_{\Tb^n}gd\mu
\end{eqnarray*}
when $f(z)=z^k$ and $g(z)=z^l$ for  $k, l$ in $\Zb^{n\times 1}$. Since the linear combinations of $z^k$ for $k\in \Zb^{n\times 1}$  are dense in $L^2(\Tb^n,\mu)$, the above is also true for all  $f, g\in L^2(\Tb^n,\mu)$. In particular it holds for $f=1_E$ and $g=1_F$ for any Borel subsets $E, F$ of $\Tb^n$, that is,
\begin{displaymath}
\lim_{j\rightarrow\infty}\mu(T_A^{-j}E\cap F)=\mu(E)\mu(F).
\end{displaymath}
\end{proof}

\section{Measure Rigidity on $\Tb^n$}

In this section we prove  Theorem~\ref{rigidity on torus} and Corollary~\ref{C-rigidity}. For this we need the following Lemma~\cite[Lemma 5.1]{H}.

\begin{lemma}\label{dirac}
Let $T:X\rightarrow X$ be a continuous map on a compact metrizable space $X$. Then a weakly mixing $T$-invariant measure $\mu$ on $X$ with an atom is always a Dirac measure, i.e. $\supp(\mu)$ is a singleton.
\end{lemma}

Note that a measure $\mu$ on $\Tb^n$ is the Lebesgue measure iff $\hat{\mu}(k)=0$ for all nonzero $k\in \Zb^{n\times 1}$.

\begin{proof}[Proof of Theorem~\ref{rigidity on torus}]\

(1) Suppose $\mu$ is an ergodic $\times A$-invariant measure on $\Tb^n$ and there exist an $n$-tuple $(B_1, B_2, \cdots, B_n)$ of matrices in $M_n(\Zb)$ which is strongly independent over $\Qb$ and a F{\o}lner sequence $\Sigma=\{F_m\}_{m=1}^\infty$ in $\Nb$  such that $\mu$ is $\times(A^j+B_i)$-invariant for every $1\leq i\leq n$ and $j$ in some $E\subseteq \mathbb{N}$ with $\overline{D}_{\Sigma}(E)=1$. Passing to a subsequence of $\Sigma$ if necessary, we may assume that $\lim_{m\rightarrow \infty}\frac{|F_m\cap E|}{|F_m|}=1$. By Lemma~\ref{lemma:inv}, we have $\hat{\mu}(A^jk+B_ik)=\hat{\mu}(k)$ for all $j\in E$, $1\le i\le n$ and $k\in \Zb^{n\times 1}$.

Assume that  $\mu$ is not the Lebesgue measure. Then there exists a nonzero $k\in \Zb^{n\times 1}$ such that $\hat{\mu}(k)\neq 0$.

Since $\mu$ is an ergodic $\times A$-invariant measure, by Theorem~\ref{fourier coefficient} (1), we have
$\lim_{m\rightarrow \infty}\frac{1}{|F_m|}\sum_{j\in F_m} \hat{\mu}(A^jk+B_ik)=\hat{\mu}(k)\hat{\mu}(B_ik)$
for every $1\leq i\leq n$. Note that
\begin{align*}
&\frac{1}{|F_m|}\sum_{j\in F_m} \hat{\mu}(A^jk+B_ik)\\
&=\frac{1}{|F_m|}\sum_{j\in F_m\cap E}\hat{\mu}(A^jk+B_ik)+\frac{1}{|F_m|}\sum_{j\in F_m\setminus E} \hat{\mu}(A^jk+B_ik)\\
&=\frac{|F_m\cap E|}{|F_m|}\hat{\mu}(k)+\frac{1}{|F_m|}\sum_{j\in F_m\setminus E} \hat{\mu}(A^jk+B_ik)\rightarrow\hat{\mu}(k)
\end{align*}
 as $m\rightarrow \infty$.  Hence $\hat{\mu}(k)=\hat{\mu}(k)\hat{\mu}(B_ik)$ which implies $\hat{\mu}(B_ik)=1$ for every $1\leq i\leq n$. From Lemma~\ref{atomic} we get that $\mu$ is finitely supported.

(2) Suppose $\mu$ is a weakly mixing $\times A$-invariant measure on $\Tb^n$ and
there exist an $n$-tuple $(B_1, B_2, \cdots, B_n)$ of matrices in $M_n(\Zb)$ which is strongly independent over $\Qb$
and a F{\o}lner sequence $\Sigma=\{F_m\}_{m=1}^\infty$ such that $\mu$ is $\times(A^j+B_i)$-invariant for every $1\leq i\leq n$ and $j$ in some $E\subseteq \mathbb{N}$ with $\overline{D}_{\Sigma}(E)>0$. By Lemma~\ref{lemma:inv}, we have $\hat{\mu}(A^jk+B_ik)=\hat{\mu}(k)$ for all $j\in E$, $1\le i\le n$ and $k\in \Zb^{n\times 1}$.

Assume that  $\mu$ is not the Lebesgue measure.  Then there exists a nonzero $k\in \Zb^{n\times 1}$ such that $\hat{\mu}(k)\neq 0$.

Let $1\le i\le n$. Since $\mu$ is a weakly mixing $\times A$-invariant measure, by Theorem~\ref{fourier coefficient} (2), we have
$\lim_{m\rightarrow \infty}\frac{1}{|F_m|}\sum_{j\in F_m}|\hat{\mu}(A^jk+B_ik)-\hat{\mu}(k)\hat{\mu}(B_ik)|^2=0$.
Therefore,
 \begin{eqnarray*}
 0&=&  \limsup_{m\rightarrow \infty}\frac{1}{|F_m|}\sum_{j\in F_m} |\hat{\mu}(A^jk+B_ik)-\hat{\mu}(k)\hat{\mu}(B_ik)|^2\\
 &\geq& \limsup_{m\rightarrow \infty}\frac{1}{|F_m|}\sum_{j\in F_m\cap E} |\hat{\mu}(A^jk+B_ik)-\hat{\mu}(k)\hat{\mu}(B_ik)|^2\\
 &=& \limsup_{m\rightarrow \infty}\frac{1}{|F_m|}\sum_{j\in F_m\cap E} |\hat{\mu}(k)-\hat{\mu}(k)\hat{\mu}(B_ik)|^2\\
 &=& |\hat{\mu}(k)-\hat{\mu}(k)\hat{\mu}(B_ik)|^2 \overline{D}_{\Sigma}(E).
\end{eqnarray*}
Hence $\hat{\mu}(k)-\hat{\mu}(k)\hat{\mu}(B_ik)=0$, which implies that $\hat{\mu}(B_ik)=1$. From Lemma~\ref{atomic} we get that $\mu$ is finitely supported.

(3) Suppose $\mu$ is a strongly mixing $\times A$-invariant measure on $\Tb^n$ and
there exist an $n$-tuple $(B_1, B_2, \cdots, B_n)$ of matrices in $M_n(\Zb)$ which is strongly independent over $\Qb$
and an infinite set $E\subseteq\mathbb{N}$ such that $\mu$ is $\times(A^j+B_i)$-invariant for every $1\leq i\leq n$ and $j$ in $E$.

Assume that $\mu$ is not the Lebesgue measure. Then there exists a nonzero $k\in \Zb^{n\times 1}$ such that $\hat{\mu}(k)\neq 0$.

Let $1\le i\le n$. Since $\mu$ is a strongly mixing $\times A$-invariant measure, by Theorem~\ref{fourier coefficient} (3), we have
\begin{eqnarray*}
\lim_{j\rightarrow \infty} \hat{\mu}(A^jk+B_ik)=\hat{\mu}(k)\hat{\mu}(B_ik).
\end{eqnarray*}
Owing to $\mu$ being $\times(A^j+B_i)$-invariant for all $j\in E$, by Lemma~\ref{lemma:inv} one has $\hat{\mu}(A^jk+B_ik)=\hat{\mu}(k)$ for all $j\in E$. Consequently, $\hat{\mu}({k})=\hat{\mu}(k)\hat{\mu}(B_ik)$, which implies $\hat{\mu}(B_ik)=1$. From Lemma~\ref{atomic} we get that $\mu$ is finitely supported.

Suppose $\mu$ is a measure on $\Tb^n$ satisfying (2) or (3) of Theorem \ref{rigidity on torus}. If $\mu$ is not a Lebesgue measure, then $\mu$ is finitely supported. According to Lemma~\ref{dirac}, we conclude that $\mu$ is a Dirac measure on $\Tb^n$.
\end{proof}

\begin{proof}[Proof of Corollary~\ref{C-rigidity}]\

Take a nonzero $B$ in $M_n(\Zb)$ with $P_B(t)$ irreducible in $\Qb[t]$~(see Remark~\ref{remark: existence}). Then $B$ is strongly independent over $\Qb$ by Theorem~\ref{strongly independent matrix}.  The multiplicative semigroup $S$ generated by $\{B, B^j+B^i\}_{0\leq i\leq n-1, j\geq 1}$, where we put $B^0=I_n$, is what we need.
\end{proof}

\section*{Acknowledgement}We thank helpful comments from Huaxin Lin, Kunyu Guo, Wenming Wu, Shengkui Ye and Yi Gu.


\begin{thebibliography}{999}

\small
\bibitem[B]{B}
T. Bewley. Extension of the Birkhoff and von Neumannn ergodic theorems to semigroup actions. {\it Ann. Inst. H. Poincar\'{e} Sect. B (N.S.)} {\bf 7} (1971), 283--291.

\bibitem[F]{F}
H. Furstenberg. Disjointness in ergodic theory, minimal sets, and a problem in Diophantine approximation. {\it Math. Systems Theory} {\bf 1} (1967), 1--49.

\bibitem[H]{H}
H. Huang. Fourier coefficients of $\times p$-invariant measures. {\it J. Mod. Dyn.} {\bf 11} (2017), 551--562.

\bibitem[J]{J}
A. S. A. Johnson. Measures on the circle invariant under multiplication by a nonlacunary subsemigroup of the integers. {\it Israel J. Math.} {\bf 77} (1992), 211--240.

\bibitem[KK]{KK}
B. Kalinin and A. Katok. Invariant measures for actions of higher rank abelian groups. {\it Smooth Ergodic Theory and its Applications (Seattle, WA, 1999)}, 593–-637, Proc. Sympos. Pure Math., {\bf 69}, Amer. Math. Soc., Providence, RI, 2001.

\bibitem[KS1]{KS1}
 A. Katok and R. J. Spatzier. Invariant measures for higher-rank hyperbolic abelian actions. {\it Ergodic Theory Dynam. Systems} {\bf 16} (1996), 751--778.

\bibitem[KS2]{KS2}
A. Katok and R. J. Spatzier. Corrections to ``Invariant measures for higher-rank hyperbolic abelian actions''. {\it Ergodic Theory Dynam. Systems} {\bf 18} (1998), 503--507.


\bibitem[La]{Lang}
S. Lang. {\it Algebra.} Revised third edition. Graduate Texts in Mathematics, {\bf 211}. Springer-Verlag, New York, 2002.

\bibitem[L]{L}
R. Lyons. On measures simultaneously 2- and 3-invariant. {\it Israel J. Math.} {\bf 61} (1988), 219--224.

\bibitem[Ro]{Rotman}
S. Roman. {\it Advanved Linear Algebra}. Third edition. Graduate Texts in Mathematics, {\bf 135}. Springer, New York, 2008.

\bibitem[R]{R}
D. J. Rudolph. $\times 2$ and $\times 3$ invariant measures and entropy. {\it Ergodic Theory Dynam. Systems} {\bf 10} (1990), 395--406.

\bibitem[W]{Walters}
P. Walters. {\it An Introduction to Ergodic Theory.} Graduate Texts in Mathematics, {\bf 79}. Springer-Verlag, New York-Berlin, 1982.

\end{thebibliography}
\end{document}